\documentclass[11pt]{article}

\usepackage[T1]{fontenc}
\usepackage{lmodern}
\usepackage[utf8]{inputenc}
\usepackage[english]{babel}

\usepackage{geometry}
\usepackage{xcolor}
\usepackage{microtype}
\usepackage{graphicx}

\usepackage{amsmath, amssymb, amsthm, mathtools, mathrsfs}
\numberwithin{equation}{section}

\renewenvironment{abstract}
  {%
   \begin{center}%
   \begin{minipage}{0.8\textwidth}%
   \small
   \noindent\textbf{ABSTRACT.} %
  }%
  {%
   \par
   \end{minipage}%
   \end{center}%
   \vspace{\baselineskip}%
  }

\usepackage{titlesec}

\titleformat{\section}
  {\normalfont\normalsize\centering}
  {\thesection.}
  {0.5em}
  {}

\titleformat{\subsection}[runin]
  {\normalfont\normalsize\bfseries}
  {\thesubsection}
  {0.5em}
  {}
  [.]

\titlespacing*{\subsection}
  {0pt}
  {1.5ex plus .2ex}
  {0.5em}

\theoremstyle{plain}
\newtheorem{theorem}{Theorem}[section]
\newtheorem{lemma}[theorem]{Lemma}

\newtheorem{corollary}[theorem]{Corollary}

\theoremstyle{definition}

\theoremstyle{remark}

\newcommand{\sumstar}{\sideset{}{^*}\sum}
\newcommand{\sumh}{\sideset{}{^h}\sum}
\newcommand{\sumd}{\sideset{}{^d}\sum}

\usepackage{hyperref}

\newcommand{\msubsection}[2]{\subsection[#1]{#1 \texorpdfstring{$#2$}{#2}}}

\let\originalleft\left
\let\originalright\right
\renewcommand{\left}{\mathopen{}\mathclose\bgroup\originalleft}
\renewcommand{\right}{\aftergroup\egroup\originalright}

\makeatletter
\renewcommand{\maketitle}{%
  \begin{center}
    {\Large \@title\par}%
    \vspace{0.75\baselineskip}%
    {\normalsize \@author\par}%
    \vspace{\baselineskip}%
  \end{center}%
}
\makeatother

\title{\textbf{ONE-LEVEL DENSITY OF ZEROS OF $\Gamma_1(q)$ $L$-FUNCTIONS}}
\author{ARIJIT PAUL}
\date{}

\begin{document}
\maketitle
\begin{abstract}
We study the one-level density of zeros for a family of $\Gamma_1(q)$ $L$-functions.
Assuming GRH, we are able to extend the support of the Fourier transform of the test function to
$\left(-\frac{8}{3},\frac{8}{3}\right)$ and verify the Katz–Sarnak prediction for our unitary family. As an application,  we obtain that the proportion
of forms in the family with non-vanishing at the central point is at least $62.5\%$, assuming GRH.
This is the highest non-vanishing proportion for any family associated with a unitary group. Moreover, this result indicates that the structural properties of $L$-functions play a more important role in extending the support than the associated symmetry group.

\end{abstract}

\medskip

\section{INTRODUCTION}
\subsection{Background and main result}
A central topic in analytic number theory is the statistical distribution of zeros of
$L$-functions. Katz and Sarnak \cite{KS} formulated the philosophy that zeros of
automorphic $L$-functions in natural families exhibit the same local statistics as eigenvalues of random matrices
from a classical compact group determined by the symmetry of the family. Inspired by their work, Iwaniec-Luo-Sarnak \cite{ILS} introduced one-level density to understand the distribution of zeros of a family of $L$-functions near the central point. In particular, for an even Schwartz function $\varphi$ with compactly supported Fourier transform, they defined the one-level density  for any natural family $\mathcal F$ of automorphic \(L\)-functions as
\[
D_{\mathcal F}(\varphi)\ :=\ \frac{1}{|\mathcal F|}\sum_{f\in\mathcal F}\ \sum_{\gamma_f}
\varphi\!\left(\frac{\gamma_f\log Q_f}{2\pi}\right),
\]
where $Q_f$ is the analytic conductor and $\rho_f=\tfrac12+i\gamma_f$ ranges over the nontrivial
zeros. Katz and Sarnak conjectured that for any natural family $\mathcal F$ of automorphic \(L\)-functions,
\begin{equation}
\lim_{Q\to\infty} D_{\mathcal F}(\varphi)\ =\ \int_{-\infty}^{\infty}\varphi(x)\,W(G)(x)\,dx,
\end{equation}
where $G\in\{U,O,Sp, SO(\text{even}), SO(\text{odd})\}$ is the symmetry group associated with $\mathcal F$ and $W(G)$ is the corresponding
one-level kernel. Several studies affirm the conjecture in particular contexts. However, in each scenario, we obtain restrictions on the support of the Fourier transform of the test function $\varphi$. This constraint arises due to limitations in our ability to
bound off–diagonal contributions.\\

An interesting problem in this area is to enlarge the admissible support of $\widehat{\varphi}$ beyond the range already available in the literature. Extending the support is important, as it gives more refined arithmetic information, such as stronger results on the proportion of non-vanishing central $L$-values and evidence toward the simplicity of zeros. Although there has been steady progress over the years, extending the support without extra conditions (except GRH) is extremely difficult. For instance, Iwaniec-Luo-Sarnak \cite{ILS} showed that, for a family of holomorphic cusp forms,
assuming GRH, one obtains a one-level density result, as predicted by Katz–Sarnak for  $\operatorname{supp} \widehat{\varphi}\subset(-2,2)$.
Under an additional hypothesis, Iwaniec-Luo-Sarnak further extended the admissible
support to $\operatorname{supp} \widehat{\varphi}\subset(-22/9,22/9)$. Later, for a unitary family of $\mathrm{GL}_1$ Dirichlet $L$-functions modulo a prime $q$, Hughes-Rudnick~\cite{HR} verified the Katz--Sarnak prediction for the unitary case with $\operatorname{supp} \widehat{\varphi}\subset[-2,2]$ by averaging over the nontrivial characters modulo prime $q$. In a similar setting, assuming GRH, Fiorilli--Miller~\cite{FM} enlarged the family and considered the dyadically averaged GL$_1$ family of Dirichlet $L$-functions modulo $q$ over $Q/2<q\le Q$ and obtained the expected main term for $\operatorname{supp} \widehat{\varphi}\subset(-4,4)$, with the help of the 'de-averaging' hypothesis. In a recent work, Drappeau-Pratt-Radziwill~\cite{DPR} confirm the Katz--Sarnak prediction 
for the unitary family of $\mathrm{GL}_1$ Dirichlet $L$-functions with characters ranging over primitive Dirichlet characters of conductor in $[Q/2, Q]$, with  $\operatorname{supp}\widehat{\varphi}\subset\!\bigl(-2-\tfrac{50}{1093},\,2+\tfrac{50}{1093}\bigr)$.  Their work is the first to push the support unconditionally beyond the “diagonal’’ range.\\

A natural question is what structural features enable an extension of the admissible support.
We notice that both in \cite{FM} and \cite{DPR}, they surpass the classical support $(-2,2)$ by averaging over a larger family. It suggests that a larger family may provide enlarged support. Recent work of Baluyot, Chandee, and Li \cite{BCL} also supports this observation. Assuming GRH, for a large orthogonal $GL_2$ family, they show that the one-level density prediction holds for the support of the Fourier transform of the test function in $(-4,4)$, a significant increase from the classical $(-2,2)$ barrier. The size of both families in \cite{DPR} and \cite{BCL} is $Q^2$, while the conductor remains around $Q$. The bandwidth limit for $\widehat \varphi$ of $(-2, 2)$ is a difficult barrier for the unitary family of $GL_1$ $L$-functions in \cite{DPR}. However, for the large orthogonal family of $GL_2$ $L$-functions in \cite{BCL}, it is possible to extend the support to $(-4, 4)$. In this paper, we investigate what factors affect how much we can extend the support. Is it completely predicted by the symmetry group associated with the family of $L$-functions, or is it an effect of the finer structure of the family of $L$-functions?
To answer this question, we study the one-level density of zeros for $\Gamma_1(q)$ $L$-functions, which is a large unitary family of $GL_2$ $L$-functions. We successfully extend the support of the Fourier transform of the test function to $(-\frac{8}{3},\frac{8}{3})$ under GRH, indicating that the $GL_2$ structure plays the decisive role in extending the support.\\

To be more precise, let $k \geq 3$ and $q$ be positive integers. For a Dirichlet character $\chi$ modulo $q$, let 
$S_k(\Gamma_0(q),\chi)$
denote the space of holomorphic cusp forms of weight $k$ for the congruence subgroup $\Gamma_0(q)$ with nebentypus character $\chi\pmod{q}$. Here
\[
\Gamma_0(q) = \left\{
\begin{pmatrix} a & b \\ c & d \end{pmatrix} \in SL_2(\mathbb{Z}) \;:\; c \equiv 0 \pmod{q}
\right\}.
\]
Within this setting, we are especially interested in the subgroup
\[
\Gamma_1(q) = \left\{
\begin{pmatrix} a & b \\ c & d \end{pmatrix} \in SL_2(\mathbb{Z}) \;:\; c \equiv 0 \pmod{q}, \ a \equiv d \equiv 1 \pmod{q}
\right\},
\]
and the corresponding space $S_k(\Gamma_1(q))$ of cusp forms of weight $k$. One has the natural decomposition
\[
S_k(\Gamma_1(q)) \;=\; \bigoplus_{\chi \bmod q} S_k(\Gamma_0(q),\chi).
\]
For a given character $\chi$, we suppose $H_k(q,\chi) \subseteq S_k(\Gamma_0(q),\chi)$ denotes an orthogonal basis of  $S_k(\Gamma_0(q),\chi)$ consisting of Hecke
cusp forms, normalized so that their first Fourier coefficient equals $1$. For $f \in H_k(q,\chi)$, write its Fourier expansion as
\[
f(z) = \sum_{n \geq 1} a_f(n) e^{2\pi i n z}, \qquad a_f(1) = 1.
\]
The associated Hecke eigenvalues $\lambda_f(n)$ are defined via $a_f(n) = \lambda_f(n) n^{(k-1)/2}$. The $L$-function attached to $f$ is then given, for $\Re(s) > 1$, by
\[
L(s,f) = \sum_{n=1}^\infty \frac{\lambda_f(n)}{n^s}
= \prod_p \left(1 - \frac{\lambda_f(p)}{p^s} + \frac{\chi(p)}{p^{2s}}\right)^{-1}
= \prod_p \left(1 - \frac{\alpha_f(p)}{p^s}\right)^{-1}
         \left(1 - \frac{\beta_f(p)}{p^s}\right)^{-1},
\]
where $\alpha_f(p), \beta_f(p)$ are complex numbers of absolute value $1$. When $f$ is a newform, $L(s,f)$ extends analytically to the entire complex plane and satisfies a functional equation of the shape
\[
\Lambda(s,f) := \left(\frac{q}{4\pi^2}\right)^{(s-\frac{1}{2})/2} \Gamma\!\left(s + \tfrac{k-1}{2}\right) L(s,f)
= \epsilon_f \, \Lambda(1-s,\overline{f}),
\]
with root number $\epsilon_f$ of absolute value $1$. Under GRH, the nontrivial zeros of $L(s,f)$ all lie on the critical line $\Re(s)=\tfrac{1}{2}$, and we denote them as $\rho_f = \tfrac{1}{2} + i\gamma_f$. We assume GRH throughout the paper. If $\alpha_f$ is any number associated to $f \in H_k(q,\chi)$, then its harmonic average is defined as
\[
\sumh_{f \in H_k(q,\chi)} \alpha_f
:= \frac{\Gamma(k-1)}{(4\pi)^{k-1}} \sum_{f \in H_k(q,\chi)} \frac{\alpha_f}{\langle f,f\rangle},
\]
where $\langle f,f\rangle$ is the Petersson inner product. Throughout this paper, we consider the Fourier transform of $f$ as 
\begin{equation*}
    \hat{f}(u)=\int_{-\infty}^{\infty} f(x) e^{-2\pi i x u} \ dx.
\end{equation*}
Let $\varphi$ be a smooth compactly supported function on $\mathbb{R}$, and define the $\Gamma_1(q)$-analogue of the one-level density as
\[
D_{\mathcal{F}}(\varphi) := \frac{2}{\phi(q)} \sum_{\substack{\chi \bmod q \\ \chi(-1)=(-1)^k}}
\sumh_{f \in H_k(q,\chi)} \sum_{\gamma_f}
\varphi\!\Big(\tfrac{\gamma_f \log q}{2\pi}\Big).
\]
Our main theorem is the following.

\begin{theorem}\label{Main Thm}
Assume GRH. Let $\varphi$ be an even Schwartz function with $\operatorname{supp}\widehat{\varphi} \subset (-\frac{8}{3},\frac{8}{3})$. Then, with the notation above, for $q$ prime and $k\geq 3$ odd,
\[
D_{\mathcal{F}}(\varphi) =\frac{2}{\phi(q)} \sum_{\substack{\chi \bmod q \\ \chi(-1)=-1}}
\sumh_{f \in H_k(q,\chi)} \sum_{\gamma_f}
\varphi\!\Big(\tfrac{\gamma_f \log q}{2\pi}\Big)
= \int_{-\infty}^{\infty} \varphi(x)\,dx + o(1).
\]
\end{theorem}
We assume $k$ to be odd so that all $f\in H_k(q,\chi)$ are newforms and that way our calculations are brief. For even 
$k$, all forms $f \in H_k(q,\chi)$ are newforms, except possibly when $\chi$ is trivial and $f$ is induced from a full-level cusp form. Removing this assumption is somewhat cumbersome but not difficult. 

For completeness, we provide the proof of Theorem \ref{Main Thm} in Section \ref{Sussection 3.2} assuming Lemmas \ref{lemma 4.2 a} and \ref{Lemma 4.3}. The main technical work is the proof of Lemmas \ref{lemma 4.2 a} and \ref{Lemma 4.3}, which comprises the bulk of the paper (see Sections \ref{Section 6} and \ref{Section 7}). Before we discuss those details, we highlight an application of Theorem \ref{Main Thm}. 
The technique developed in this paper can be extended to study higher-level densities of the same family with the same support, which the author and Chandee are currently investigating.

\subsection{Application}
As a direct application, we can use our result to find the proportion of non-vanishing at the central point for this large unitary family. From the following result, we get that the proportion of forms in the family with non-vanishing at the central point is at least $62.5\%$, assuming GRH. This is the record non-vanishing proportion associated to any unitary family so far.
\begin{corollary}
Assume GRH. The harmonic proportion of non-vanishing $\Gamma_1(q)$ $L$-functions is at least $\frac{5}{8}$, in the sense that
\begin{equation}
\liminf_{q\to\infty}
\frac{2}{\phi(q)}
\sum_{\substack{\chi \bmod q \\ \chi(-1)=-1}}
\sumh_{f\in H_k(q,\chi)}
\mathbf{1}_{\{L(\frac12,f)\neq 0\}}
\ge \frac{5}{8}.
\end{equation}
\end{corollary}

\begin{proof}
Applying Theorem~\ref{Main Thm} and Theorem~2(i) of \cite{CCM}, we obtain
\[
\limsup_{q\to\infty}
\frac{2}{\phi(q)}
\sum_{\substack{\chi \bmod q \\ \chi(-1)=-1}}
\sumh_{f\in H_k(q,\chi)}
\operatorname{ord}_{s=\frac12}L(s,f)
\le \frac{1}{K(0,0)},
\]
where $K(w,z)$ is the reproducing kernel in the unitary case. Since the support is
\[
\operatorname{supp}\widehat{\varphi} \subset (-\frac{8}{3},\frac{8}{3}),
\]
we have
\[
K(w,z)=\frac{\sin\!\left(\frac{8\pi}{3}(z-\overline{w})\right)}{\pi (z-\overline{w})},
\]
and hence
\[
K(0,0)=\frac83.
\]
Therefore,
\[
\limsup_{q\to\infty}
\frac{2}{\phi(q)}
\sum_{\substack{\chi \bmod q \\ \chi(-1)=-1}}
\sumh_{f\in H_k(q,\chi)}
\operatorname{ord}_{s=\frac12}L(s,f)
\le \frac38.
\]

Now, for each $f$,
\[
\mathbf{1}_{\{L(\frac12,f)=0\}}
\le \operatorname{ord}_{s=\frac12}L(s,f).
\]
Hence,
\[
\limsup_{q\to\infty}
\frac{2}{\phi(q)}
\sum_{\substack{\chi \bmod q \\ \chi(-1)=-1}}
\sumh_{f\in H_k(q,\chi)}
\mathbf{1}_{\{L(\frac12,f)=0\}}
\le \frac38.
\]
Taking complements gives
\[
\liminf_{q\to\infty}
\frac{2}{\phi(q)}
\sum_{\substack{\chi \bmod q \\ \chi(-1)=-1}}
\sumh_{f\in H_k(q,\chi)}
\mathbf{1}_{\{L(\frac12,f)\neq 0\}}
\ge 1-\frac38=\frac58.
\]
This completes the proof.
\end{proof}
We now briefly discuss the outline of the proof before proceeding with the proof of Theorem \ref{Main Thm}.

\subsection{Outline of the Proof}
After applying the explicit formula, it remains to bound the following off-diagonal terms.
\begin{equation*}
    \frac{2}{\phi(q)}\sum_{\substack{\chi(q) \\ \chi(-1)=-1 }} \sumh_{f \in H_k(q,\chi)}\sum_pa_p\lambda_f(p),
\end{equation*}
where
\begin{equation*}
    a_n = \frac{\Lambda(n)}{\sqrt{n}} \hat{\varphi}\left( \frac{\log n}{\log q} \right).
\end{equation*}
Using orthogonality and rearranging the exponential sum, we obtain 
\begin{equation*}
    \mathcal{K} \sum_{s,t}\frac{2\pi}{qst}\sumstar_{m(t)} \sumstar_{\substack{x(t) \\ (x+qs,t)=1 }}e\left(\frac{m\overline{x}-n\overline{x+qs}}{t}\right)\sum_{\substack{p \\ p\equiv m(t) }} a_p e\left(\frac{p+1}{qst}\right)J_{k-1}\left(\frac{4\pi \sqrt p}{qst}\right).
\end{equation*}
We express $p\equiv m\ (\bmod \ t)$ using Dirichlet characters. To illustrate the idea of the proof we consider the transition region $st\sim \frac{\sqrt p}{q}$ with $p\sim P$.
Considering the trivial character, we essentially get
\begin{equation*}
S_1:=\sum_{st \sim \frac{\sqrt P}{q}}\frac{1}{qst}\sum_p \frac{\log p}{\sqrt p}\ e\Bigg(\frac{p}{qst}\Bigg) w\Bigg(\frac{p}{P}\Bigg).
\end{equation*}
We use Mellin transform for $e\bigg(\frac{p}{qst}\bigg)w\bigg(\frac{p}{P}\bigg)$ to get
\begin{equation*}
 \int_{(1)} \sum_{st \sim\frac{\sqrt p}{q}}  \frac{1}{qst}\sum_{p} \frac{\log p}{p^{1/2+z}}\tilde{w}(z)\ dz,
\end{equation*}
where
\[
   \tilde{w}(z)\ll
   \begin{cases}
      \sqrt{\tfrac{qst}{P}}, & \text{if } |\Im (z)|\sim \frac{P}{qst}, \\[6pt]
      \dfrac{1}{\left(|\Im (z)|+\tfrac{P}{qst}\right)^A}, & \text{otherwise},
   \end{cases}
\]
Assuming GRH, we move the contour to $\Re(z)>0$. From the bound on $\widetilde{w}(z)$ we obtain
$S_1 \ll \frac{P^{1/4}}{q};$
in particular, $S_1$ is small provided $P \ll q^{4-\delta}$.\\

For non-trivial characters, we apply Cauchy-Schwarz and the orthogonality relation of Dirichlet characters, and essentially we get
\begin{equation*}
    S_2:=\frac{1}{q}\sum_{st\sim \frac{\sqrt P}{q}}\frac{1}{s}\Bigg(\sum_{\substack{\chi(\bmod t) \\ \chi\neq\chi_0 }}\Bigg |\sum_{p}\frac{\chi(p)\log p}{\sqrt p}e\Bigg(\frac{p}{qst}\Bigg)\psi\Bigg(\frac{p}{P}\Bigg)\Bigg|^2\Bigg)^{\frac{1}{2}}.
\end{equation*}
Again, using the Mellin transform for $e\bigg(\frac{p}{qst}\bigg)\psi\bigg(\frac{p}{P}\bigg)$ and similar bounding techniques as before, assuming GRH, we get
$S_2\ll \frac{P^{\frac{3}{4}}}{q^2},$ which is small when $P\leq q^{\frac{8}{3}-\delta}$.

\section{PRELIMINARY LEMMAS}

In this section, we collect several lemmas that will be used later in the paper. We begin
with the orthogonality relations for Dirichlet characters modulo $q$.

\begin{lemma}
Let $q\ge 3$. Then for integers $m,n$,
\begin{equation}\label{Lemma Orthogonality of Ch}
\frac{2}{\phi(q)}\sum_{\substack{\chi\bmod q\\ \chi(-1)=-1}} \chi(m)\overline{\chi(n)}
=
\begin{cases}
1, & m\equiv n\pmod q \text{ and }(mn,q)=1,\\[2mm]
-1, & m\equiv -n\pmod q \text{ and }(mn,q)=1,\\[2mm]
0, & \text{otherwise.}
\end{cases}
\end{equation}
\end{lemma}
Define the operator $\mathcal{K}$ by
\[
\mathcal{K}f = i^{-k} f + i^{k} \overline{f} = 2\Re\big(i^{-k} f\big).
\]
Next, we record a Petersson-type formula suitable for summing over the family. This gives an expression for sums over Dirichlet characters and normalized Fourier coefficients of cusp forms.

\begin{lemma}\label{Lemma Peterson type formula}
For any positive integers $m,n$, define
\begin{equation*}
\Delta(m,n):=\frac{2}{\phi(q)}\sum_{\substack{\chi\bmod q\\\chi(-1)=-1}}
\sumh_{f\in H_k(q,\chi)} \overline{\lambda_f(m)}\lambda_f(n).
\end{equation*}
Then
\begin{equation}\label{Equation Peterson type formula}
\Delta(m,n)=\delta(m,n)+\sigma(m,n),
\end{equation}
where
\[
\sigma(m,n)
=\mathcal{K}\sum_{s}\sum_{t}\frac{2\pi}{qst}\,V_{qs}(m,n;t)\,e\!\Big(\frac{m+n}{qst}\Big)
J_{k-1}\!\Big(\frac{4\pi\sqrt{mn}}{qst}\Big),
\]
and
\[
V_{qs}(m,n;t)=\sum_{\substack{x\bmod t\\ (x(x+qs),t)=1}} e\!\Big(\frac{m\overline{x}-\overline{nx+qs}}{t}\Big).
\]
Moreover,
\begin{equation}  \label{Equation bound on peterson type formula}
\Delta(m,n)=\delta(m,n)+O\!\Big(\frac{\sqrt{mn}}{q^2}\Big). 
\end{equation}
\end{lemma}
The proof of Equation \eqref{Equation Peterson type formula} can be found in Lemma 2.1 of \cite{IL}, while the proof of Equation \eqref{Equation bound on peterson type formula} is given in Equation (2.14) of \cite{IL}.\\

We now state a lemma summarizing some standard formulas related to the $J$-Bessel function.

\begin{lemma}\label{Lemma J-Bessel Function}
Let $J_{k-1}$ be the $J$-Bessel function of order $k-1$. Then
\begin{equation}\label{J-Bessel expansion 1}
    J_{k-1}(2\pi x) = \frac{1}{2\pi \sqrt{x}}
    \left[w_k(2\pi x)\,e\!\left(x-\frac{k}{4}+\frac{1}{8}\right)
    + \overline{w}_k(2\pi x)\,e\!\left(-x+\frac{k}{4}-\frac{1}{8}\right)\right],
\end{equation}
where $w_k^{(j)}(x)\ll_{j,k}x^{-j}$. Moreover,
\begin{equation}\label{J-Bessel expansion 2}
    J_{k-1}(2x)=\sum_{l=0}^{\infty}(-1)^l \,
    \frac{x^{2l+k-1}}{l!\,(l+k-1)!},
\end{equation}
and
\begin{equation} \label{bound of J bessel}
    J_{k-1}(x)\ll \min\{x^{-\frac{1}{2}},\,x^{k-1}\}.
\end{equation}
\end{lemma}

\begin{proof}
See \cite{W} for the details of the proof.
\end{proof}

The next lemma provides a bound for certain prime sums twisted by Dirichlet characters.  
This estimate will be used to control the error terms in Sections \ref{Section 6} and \ref{Section 7}, once the congruence condition has been removed by insertion of characters.

\begin{lemma}\label{Lemma 2.4}
Assume the Generalized Riemann Hypothesis (GRH) for $L(s,\chi)$ with $\chi \bmod q$.  
Let $\psi$ be a smooth function supported in $(a,b)$, where $0<a<b$.  
Write $z=\sigma+it$ with 
\[
\sigma \in \left(\tfrac{1}{2}-\tfrac{10}{\log q}, \, \tfrac{1}{2}+\tfrac{10}{\log q}\right), \qquad t \in \mathbb{R}.
\]
If $\chi$ is a nontrivial character, then
\begin{equation*}
    \sum_p \frac{\chi(p)\log p \,\psi(p/X)}{p^z}
    \ll \log^2(q+|t|)\,
    \max_{a\leq x \leq b} |\psi^{(3)}(x)|.
\end{equation*}
If $\chi$ is the trivial character, then for any $A>0$,
\begin{equation*}
    \sum_p \frac{\chi(p)\log p \,\psi(p/X)}{p^z}
    \ll \log^2(q+|t|)\,
    \max_{a\leq x \leq b} |\psi^{(3)}(x)|
    + \frac{\sqrt{X}}{(1+|t|)^A}\,
    \max_{a\leq x \leq b} |\psi^{(A)}(x)|.
\end{equation*}

\end{lemma}

\begin{proof}
See Lemma 2.3 of \cite{CLL}.
\end{proof}

We conclude this section with the explicit formula, which relates the zeros of automorphic $L$-functions to sums over prime powers.  
This formula is the starting point for the proof of Theorem \ref{Main Thm}.

\begin{lemma}\label{Lemma Explicit Formula}
Let $\varphi$ be an even Schwartz function whose Fourier transform has compact support. Then
\begin{equation*}
    \sum_{\gamma_f}\varphi\!\left(\frac{\gamma_f}{2\pi}\log q\right)
    = -\frac{1}{\log q}\sum_{n=1}^{\infty}
      \frac{\Lambda(n)\,[C_f(n)+C_{\bar{f}}(n)]}{\sqrt n}\,
      \hat{\varphi}\!\left(\frac{\log n}{\log q}\right)
      + \int_{-\infty}^{\infty}\varphi(x)\,dx
      + O_k\!\left(\frac{1}{\log q}\right),
\end{equation*}
where \\  \begin{equation}\label{C_f}
     C_f(n)=\begin{cases}
\alpha_f(p)^l+\beta_f(p)^l , & \text{ if } n=p^l,\\[2mm]
0, & \text{otherwise.}
\end{cases}
 \end{equation}

\begin{proof}
See Lemma 2.5 of \cite{BCL}.
\end{proof}
\end{lemma}

\section{PROOF OF THEOREM \ref{Main Thm}} \label{Section 4}
\subsection{Initial setup}
Applying the explicit formula from Lemma \ref{Lemma Explicit Formula} to Theorem \ref{Main Thm}, we obtain the following.
\begin{align}
    \mathcal{D_F(\varphi)}=\frac{2}{\phi(q)}\hspace{-.3 cm}\sum_{\substack{\chi \bmod q \\ \chi(-1)=-1}}
    \sumh_{f \in H_k(q,\chi)} 
    \hspace{-.1 cm}
       &\Bigg[ -\frac{1}{\log q}\sum_{n=1}^{\infty}
        \frac{\Lambda(n)\big(C_f(n)+C_{\bar f}(n)\big)}{\sqrt{n}}
        \, \hat{\varphi}\!\left(\frac{\log n}{\log q}\right) \nonumber \\
    & + \int_{-\infty}^{\infty}\varphi(x)\, dx
        + O_k\!\left(\frac{1}{\log q}\right) \Bigg]\label{Equation 4.1}
    .
\end{align}
By Equation (\ref{Equation bound on peterson type formula}), the contribution from the integral is
\begin{equation}\label{equation 4.2 a}
    \frac{2}{\phi(q)}
    \sum_{\substack{\chi \bmod q \\ \chi(-1)=-1}}
    \sumh_{f \in H_k(q,\chi)}
    \int_{-\infty}^{\infty}\varphi(x)\, dx
    \;= \int_{-\infty}^{\infty} \varphi(x)\ dx+O\left(\frac{1}{q^2}\right).
\end{equation}
We are left with the sum over prime powers in Equation ({\ref{Equation 4.1}}) which will contribute to the error term. It suffices to consider the term with $C_f(n)$, since the part with $C_{\overline{f}}(n)$ follows similarly.
Thus, our goal is to bound
\begin{equation*}
    \frac{2}{\phi(q)} \sum_{\substack{\chi \bmod q \\ \chi(-1) = -1}} 
    \sumh_{f \in H_k(q, \chi)} 
    \left[ - \frac{1}{\log q} \sum_{n=1}^{\infty} \Lambda(n) \frac{C_f(n)}{\sqrt{n}} 
    \hat{\varphi}\left( \frac{\log n}{\log q} \right) \right].
\end{equation*}
We split the sum over $n$ into prime powers and treat each case separately.
\begin{align}
    \sum_{n=1}^{\infty} \Lambda(n) \frac{C_f(n)}{\sqrt{n}} \hat{\varphi}\left( \frac{\log n}{\log q} \right) 
    &= \sum_{p} \frac{\log p \, C_f(p)}{\sqrt{p}} \hat{\varphi}\left( \frac{\log p}{\log q} \right) \nonumber + \sum_{p} \frac{\log p \, C_f(p^2)}{p} \hat{\varphi}\left( \frac{\log p^2}{\log q} \right) \nonumber \\
    &\quad + \sum_{p} \sum_{b \geq 3} \frac{\log p \, C_f(p^b)}{p^{\frac{b}{2}}} \hat{\varphi}\left( \frac{\log p^b}{\log q} \right).
\end{align}
We first deal with the contributions of terms with $n = p^b$ for $b \ge 3$.  
Since $|C_f(p^b)| \le 2$, we have
\begin{align*}\label{p^k term}
    \sum_{p} \sum_{b \ge 3} \frac{\log p \, C_f(p^b)}{p^{\frac{b}{2}}} \hat{\varphi}\left( \frac{\log p^b}{\log q} \right) 
    &\ll \sum_{\substack{p^b \ll q^{\delta} \\ b \ge 3}} \frac{2 \log p}{p^{\frac{b}{2}}} \left| \hat{\varphi}\left( \frac{\log p^b}{\log q} \right) \right|  \ll 1.
\end{align*}
We restricted $p\ll q^{\delta}$ assuming $\operatorname{supp}  \hat\varphi\subset (-\delta,\delta)$. Thus, the corresponding contribution is
\begin{equation}\label{equation 4.6 a}
    \frac{2}{\phi(q)} \sum_{\substack{\chi \bmod q \\ \chi(-1) = -1}} 
    \sumh_{f \in H_k(q, \chi)} 
    \left[ - \frac{1}{\log q} \sum_{p} \sum_{b \ge 3} \frac{\log p \, C_f(p^b)}{p^{\frac{b}{2}}} 
    \hat{\varphi}\left( \frac{\log p^b}{\log q} \right) \right] \ll \frac{1}{\log q}.
\end{equation}
Next, we consider the contribution of the terms with $n = p^2$.  Define
\begin{align}
    S_{sq} := \frac{2}{\phi(q)} \sum_{\substack{\chi \bmod q \\ \chi(-1) = -1}} 
    \sumh_{f \in H_k(q, \chi)} 
    \Bigg[ - \frac{1}{\log q} \sum_{p} \frac{\log p \, (\lambda_f(p^2) - \chi(p))}{p} 
    \hat{\varphi}\left( \frac{\log p^2}{\log q} \right) \Bigg].
\end{align}
We can restrict $p\ll q^{\delta/2}$ assuming $\operatorname{supp}  \hat\varphi\subset (-\delta,\delta)$. Now, using Equation (\ref{Equation bound on peterson type formula}) of Lemma \ref{Lemma Peterson type formula}, we bound $S_{sq}$ as
\begin{align*}\label{p^2 term}
    S_{sq} &\ll \frac{1}{\log q} \sum_{p \ll q^{\delta/2}} \frac{\log p}{p} \ \frac{p}{q^2} 
    + \frac{1}{\log q} \sum_{\substack{p \ll q^{\delta/2} \\ p \equiv 1 \, (q)}} \frac{\log p}{p}.
\end{align*}
By the Brun-Titchmarsh Theorem and the Prime Number Theorem, we get
\begin{equation}\label{Equation S_sq}
    S_{sq} \ll \frac{1}{\log q}, \hspace{1 cm} \text{provided $\delta < 4$. }
\end{equation}
Now we consider the contribution of the terms with $n = p$, which is
\begin{align}\label{equation 4.10 a}
    S_N &:= \frac{1}{\log q} \frac{2}{\phi(q)} \sum_{\substack{\chi \bmod q \\ \chi(-1) = -1}} 
    \sumh_{f \in H_k(q, \chi)} \sum_p a_p \lambda_f(p),
\end{align}
where
\begin{equation*}
    a_n = \frac{\Lambda(n)}{\sqrt{n}} \hat{\varphi}\left( \frac{\log n}{\log q} \right).
\end{equation*}
By Equation (\ref{Equation Peterson type formula}), we obtain that
\begin{equation}\label{Equation S_N}
    S_N = \frac{1}{\log q} \sum_p a_p \, \mathcal{K} \sum_{s,t} \frac{2\pi}{qst} 
    e\left( \frac{p+1}{qst} \right) V_{qs}(p,1;t) J_{k-1}\left( \frac{4\pi \sqrt{p}}{qst} \right) \\
    =: M_{off} + \epsilon_{off},
\end{equation}
where
\begin{equation*}
    V_{qs}(p,1;t) = \sum_{\substack{x \bmod t \\ (x(x+qs), t) = 1}} 
    e\left( \frac{p \bar{x} - \overline{x+qs}}{t} \right).
\end{equation*}
Here, $M_{off}$ is the contribution from terms with $(t,p) = 1$, and $\epsilon_{off}$ is the remaining part. The goal is to isolate terms with $(t,p) = 1$ and subsequently rewrite the congruence condition modulo $t$ in terms of Dirichlet characters. We will show that $\epsilon_{off}$ is small. Hence, the rest of the paper will focus on bounding the $M_{off}$ term.

\begin{lemma}
For $k \ge 3$ odd, we have
\begin{equation*}
    \epsilon_{off} = \frac{1}{\log q} \sum_p a_p \, \mathcal{K} \sum_{\substack{s,t \\ p \mid t}} \frac{2\pi}{qst} 
    e\left( \frac{p+1}{qst} \right) V_{qs}(p,1;t) J_{k-1}\left( \frac{4\pi \sqrt{p}}{qst} \right) \ll \frac{1}{q^k}.
\end{equation*}
\end{lemma}

\begin{proof}
Since $p|t$, using the change of variable $t \to tp$, we get
\begin{equation*}
    \epsilon_{off} = \frac{1}{\log q} \sum_p \frac{a_p}{p} \, \mathcal{K} \sum_{s,t} \frac{2\pi}{qst} 
    e\left( \frac{p+1}{qstp} \right) V_{qs}(p,1;tp) J_{k-1}\left( \frac{4\pi}{qst \sqrt{p}} \right).
\end{equation*}
Note that $|V_{qs}(p,1;tp)| \ll tp$. Using Equation (\ref{bound of J bessel}) of Lemma \ref{Lemma J-Bessel Function} to bound the $J$-Bessel function, we have
\begin{equation*}
    J_{k-1}(x) \ll \min \{ x^{-\frac{1}{2}}, x^{k-1} \}.
\end{equation*}
Since $\frac{4\pi}{qst \sqrt{p}} \ll 1$ and $k\geq 3$ odd, we obtain
\begin{align}
    \epsilon_{off} &\ll \frac{1}{\log q} \sum_p \frac{\log p}{\sqrt{p}} \sum_s \frac{1}{qs} \sum_t \left( \frac{4\pi}{qst \sqrt{p}} \right)^{k-1} \\
    &\ll \sum_{p \ll q^\delta} \frac{\log p}{p^{k/2}} \sum_s \frac{1}{q^k s^k} \sum_t \frac{1}{t^{k-1}} \ll \frac{1}{q^k}. \nonumber
\end{align}
\end{proof}
Therefore,
\begin{equation}\label{eqnn 3.10}
    S_N=M_{off}+O\left(\frac{1}{q^k}\right),
\end{equation}
where
    \begin{align*}
        M_{off}&:=\frac{1}{\log q} \mathcal{K}\sum_{s,t} \frac{2\pi }{qst} \sum_{\substack{p \\ (p,t)=1 }} a_p  e\left(\frac{p+1}{qst}\right)V_{qs}(p,1;t)J_{k-1}\left(\frac{4\pi\sqrt p}{qst}\right)\\&= \frac{1}{\log q}\mathcal{K} \sum_{s,t} \frac{2\pi}{qst}\sumstar_{m\bmod t}V_{qs}(m,1;t)\sum_{\substack{p \\ p\equiv m(t) }}a_p e\left(\frac{p+1}{qst}\right)J_{k-1}\left(\frac{4\pi\sqrt p}{qst}\right).
        \end{align*}
Now we express the condition $p\equiv m(t)$ using Dirichlet characters. In particular,     
        \begin{align}
        M_{off}&=\frac{1}{\log q} \mathcal{K} \sum_{s,t} \frac{2\pi}{qst}\sumstar_{m\bmod  t}V_{qs}(m,1;t)\frac{1}{\phi(t)}\sum_{\chi(t)}\bar{\chi}(m)\sum_p \chi(p)a_p e\left(\frac{p+1}{qst}\right)J_{k-1}\left(\frac{4\pi\sqrt p}{qst}\right)\nonumber \\ &
        =: S_1+S_2,\label{Equation M_{off}}
    \end{align}
    where $S_1$ is the contribution from the trivial character, and $S_2$ is the rest.
We will show the bound for $S_1$(Lemma \ref{lemma 4.2 a}) in  Section \ref{Section 6} and the bound for $S_2$(Lemma \ref{Lemma 4.3}) in Section \ref{Section 7}.
\begin{lemma}\label{lemma 4.2 a}
    With the notations as above, for $k\geq 3$ odd, assuming GRH, 
\begin{align}\label{Lemma 4.2}
    S_1\ll \frac{1}{\log q}
\end{align}
for $\delta < 4$, with $\operatorname{supp}(\hat\varphi) \subset (-\delta,\delta)$.

    \end{lemma} 
\begin{lemma}\label{Lemma 4.3}
With the notations as above, for $k\geq 3$ odd, assuming GRH, 
   \begin{equation}
    S_2\ll \frac{1}{\log q}
\end{equation}
for $\delta < \frac{8}{3}$, with $\operatorname{supp}(\hat\varphi) \subset (-\delta,\delta)$. 
\end{lemma}

\subsection{Proof of Theorem \ref{Main Thm}}\label{Sussection 3.2}

We provide a brief outline of the proof, highlighting the main steps and references to key lemmas and equations. By Equations (\ref{Equation 4.1}) to (\ref{equation 4.10 a}) and (\ref{eqnn 3.10}) we get
\begin{equation}
    \mathcal{D_F(\varphi)}= \int_{-\infty}^{\infty} \varphi(x)\ dx + M_{off} + O\left(\frac{1}{\log q}\right)+O\left(\frac{1}{q^k}\right).
\end{equation}
Now, using Equation (\ref{Equation M_{off}}) and Lemmas \ref{lemma 4.2 a} and \ref{Lemma 4.3} we get 
\begin{equation}
     \mathcal{D_F(\varphi)}=\int_{-\infty}^{\infty}\varphi(x)\,dx + O\left(\frac{1}{\log q}\right),
\end{equation}
with the support of $\hat\varphi$ in $(-\frac{8}{3},\frac{8}{3})$.  
    Hence, Theorem \ref{Main Thm} follows.

\section{PROOF OF LEMMA (\ref{lemma 4.2 a}) -- CONTRIBUTION FROM THE TRIVIAL CHARACTER}\label{Section 6}
Recall that \vspace{-.3 cm}
\begin{align*}
    S_1=\frac{1}{\log q}\,\mathcal{K} \sum_{s,t}\frac{2\pi}{qst}\sumstar_{m(t)}V_{qs}(m,1;t)\frac{1}{\phi(t)}\sum_{\substack{p \\ (p,t)=1 }} a_p e\left(\frac{p+1}{qst}\right)J_{k-1}\left(\frac{4\pi\sqrt p}{qst}\right).
\end{align*}
By using Ramanujan's sum, we get
\begin{align*}
    \sumstar_{m(t)}V_{qs}(m,1;t)
    &=\sum_{\substack{x(t) \\ (x(x+qs),t)=1 }} \sumstar_{m(t)} e\left(\frac{m\bar{x}}{t}\right) e\left(\frac{-\overline{x+qs}}{t}\right)
    =\mu(t) V_{qs}(0,1;t).
\end{align*}
Therefore, it suffices to consider
\begin{align*}
    S_1=\frac{1}{\log q}\sum_s \sum_{\substack{t }} \frac{2\pi \mu(t)}{qst \phi(t)} V_{qs}(0,1;t) \sum_{\substack{p \\ (p,t)=1 }} a_p \mathcal{K} e\left(\frac{p+1}{qst}\right) J_{k-1}\left(\frac{4\pi\sqrt p}{qst}\right).
\end{align*}
In the following lemma, we extend $S_1$ to include contributions from higher prime powers of the form
\[
a_{p^j} \, \mathcal{K} \, e\left(\frac{p^j+1}{qstd}\right) J_{k-1}\left(\frac{4\pi \sqrt{p^j}}{qstd}\right), \quad j\ge 2,
\] 
and show that their total contribution is negligible. This step is crucial for later rewriting the sum $\sum_{n\ge 1} \frac{\Lambda(n)\chi(n)}{n^s}$ as $-\frac{L'}{L}(s,\chi)$.
\begin{lemma}\label{Lemma E_{off}}
For $k \ge 3$ odd, we have
\begin{equation*}
\epsilon_{\text{higher power}} := \sum_s \sum_{t} \frac{2\pi \mu(t)}{qst \phi(t)} 
V_{qs}(0,1;t)
\sum_{\substack{p^j \\ j \ge 2\\ (p,t)=1}} a_{p^j} \mathcal{K} e\left(\frac{p^j+1}{qst}\right) 
J_{k-1}\left(\frac{4\pi \sqrt{p^j}}{qst}\right)=O\left(\frac{(\log q)^2}{q}\right).
\end{equation*}
\end{lemma}

\begin{proof}
Using Equation (\ref{bound of J bessel}) and the trivial bound $|V_{qs}(0,1;t)| \le \phi(t)$, we get
\begin{align*}
\epsilon_{\text{higher power}} 
&\ll \frac{1}{q} \sum_{\substack{j \ge 2 \\ j \ll \log q}} \sum_{p \ll q^{\delta/j}} \frac{\log p}{p^{\frac{j}{2}}} 
\sum_{t} \frac{1}{t} \sum_s \frac{1}{s} 
\min \Biggl\{ \left(\frac{\sqrt{p^j}}{qst}\right)^{-1/2}, \left(\frac{\sqrt{p^j}}{qst}\right)^{k-1} \Biggr\} \\
&\ll \frac{1}{q} \sum_{\substack{j \ge 2 \\ j \ll \log q}} \sum_{p \ll q^{\delta/j}} \frac{\log p}{p^{\frac{j}{2}}} 
\Biggl[ \sum_{\substack{s,t \\ st < \frac{\sqrt{p^j}}{q}}} \frac{1}{st} \left(\frac{\sqrt{p^j}}{qst}\right)^{-1/2} 
+ \sum_{\substack{s,t \\ st \ge \frac{\sqrt{p^j}}{q}}} \frac{1}{st} \left(\frac{\sqrt{p^j}}{qst}\right)^{k-1} \Biggr].
\end{align*}
By partial summation, the first sum is bounded by
\begin{equation*}
\left(\frac{\sqrt{p^j}}{q}\right)^{-1/2} \sum_{\substack{s,t \\ st < \frac{\sqrt{p^j}}{q}}} \frac{1}{(st)^{1/2}} 
\ll \left(\frac{\sqrt{p^j}}{q}\right)^{-1/2}  \sum_{u \ll \frac{\sqrt{p^j}}{q}} \frac{d(u)}{\sqrt{u}} \ll \log q,
\end{equation*} 
and a similar calculation gives 
    \begin{equation*}
        \sum_{\substack{s,t \\ st \ge \frac{\sqrt{p^j}}{q}}}\frac{1}{st}\left(\frac{\sqrt {p^j}}{qst}\right)^{{k-1}}\ll \log q.
    \end{equation*}
Hence,
\begin{equation*}
\epsilon_{\text{higher power}} \ll \frac{\log q}{q} \sum_{\substack{j \ge 2 \\ j \ll \log q}} \sum_{p \ll q^{\delta/j}} \frac{\log p}{p^{\frac{j}{2}}}.
\end{equation*}
Using the Prime Number Theorem and partial summation, we conclude
\begin{equation*}
\epsilon_{\text{higher power}} \ll \frac{(\log q)^2}{q}.
\end{equation*}
\end{proof}
Therefore, we obtain
\begin{equation*}
S_1 = \frac{1}{\log q} \sum_s \sum_t \frac{2\pi \mu(t)}{qst \phi(t)} V_{qs}(0,1;t) 
\sum_{\substack{r \\ (r,t)=1}} a_r \mathcal{K} e\left(\frac{r+1}{qst}\right) 
J_{k-1}\left(\frac{4\pi \sqrt r}{qst}\right) + O\left(\frac{\log q}{q}\right).
\end{equation*}
Now, by Lemma~\ref{Lemma E_{off}}, we can extend the prime sum in $S_1$ to all positive integers $r$ co-prime to $t$.\\

To handle the oscillatory behavior of the $J$-Bessel function and the exponential term, we carry out a dyadic decomposition.  
Fix a nice smooth function $W \in C_c^\infty([\frac{3}{4},2])$ with $W(x)=1$ on $[1,\frac{3}{2}]$ and $W(x)+W(\frac{x}
{2})=1$ for all $x\in [1,3]$.  Then we have 
\begin{equation}\label{4.111}
  \sumd_{R} W\!\left(\frac{x}{R}\right) = 1
  \qquad \text{for all } x\ge1.
\end{equation} 
For details, see Lemma 2.2 of \cite{KMSS}. Therefore
\[
   \sum_{r \geq 1} a_r \;=\; \sum_{r\geq 1} a_r \sumd_{R} \;W\left(\frac{r}{R}\right)\;=\; \sumd_{R}\sum_{r\sim R}a_r W\left(\frac{r}{R}\right),
\]
where $r\sim R$ means $R\leq r\leq 2R$. We now decompose $S_1$ into three ranges depending on the transition behavior of the $J$-Bessel function and the exponential term:
\begin{equation}
   S_1 \;=\; S_{1,\mathrm{case}_1} \;+\; S_{1,\mathrm{case}_2} \;+\; S_{1,\mathrm{case}_3},
\end{equation}
where $ S_{1,\mathrm{case}_1}$  is restricted to region $\frac{R}{q}\ll st$, $S_{1,\mathrm{case}_2}$ is restricted to region $\frac{\sqrt R}{q}\ll st\ll \frac{R}{q}$ and lastly $S_{1,\mathrm{case}_3}$ is restricted to $\frac{\sqrt R}{q}\gg st$. Since $\operatorname{supp} \hat{\varphi} \subset (-\delta,\delta)$, 
we can assume $R \ll q^{\delta}$ in all the three cases. The following lemma immediately implies Equation (\ref{Lemma 4.2}).
\begin{lemma}
   \label{Lemma 6.2}
   Assume GRH. Suppose that $\operatorname{supp} \hat{\varphi} \subset (-\delta,\delta)$ with $\delta < 4$. 
   Then, with the above notation, for $k\geq 3$ odd, we have the following estimates: \begin{align}
      S_{1,\mathrm{case}_1} &\ll \frac{1}{q}, \label{Equation 6.2}\\
      S_{1,\mathrm{case}_2} &\ll \frac{1}{\log q}, \label{Equation 6.3}\\
      S_{1,\mathrm{case}_3} &\ll \frac{1}{\log q}. \label{Equation 6.4}
   \end{align}
\end{lemma}
\msubsection{Proof of Equation (\ref{Equation 6.2}) — Bound for}{S_{1,\mathrm{case}_1}}
\begin{proof}

Explicitly,
\begin{equation*}
     S_{1,\mathrm{case}_1}
   := \frac{1}{\log q}\,\mathcal{K}
   \sumd_{R \ll q^\delta}\;\sum_{\substack{s,t \\ \frac{R}{q}\ll st}}\;
   \frac{2\pi \mu(t)}{qst \phi(t)} V_{qs}(0,1;t)
   \sum_{\substack{r \\ (r,t)=1}} a_r \, W\!\left(\tfrac{r}{R}\right)
   e\!\left(\tfrac{r+1}{qst}\right)
   J_{k-1}\!\left(\tfrac{4\pi \sqrt{r}}{qst}\right).
\end{equation*}
To bound this expression, we use the prime number theorem together with the bounds for the $J$-Bessel function in Equation (\ref{bound of J bessel}). Also note that $|V_{qs}(0,1;t)| \leq \phi(t)$.
Hence, for $k\geq 3$ odd and $\delta<4$, 
\begin{align*}
    S_{1,\mathrm{case}_1} 
    &\ll \frac{1}{\log q}
       \sumd_{R\ll q^\delta}
       \sum_{\substack{s,t \\ \frac{R}{q} \ll st }}
       \frac{1}{qst} \cdot \frac{\phi(t)}{\phi(t)} 
       \cdot \sqrt R
       \left(\frac{\sqrt{R}}{qst}\right)^{k-1} \\
    &\ll \frac{1}{\log q}
       \sumd_{R\ll q^\delta}
       \frac{1}{q^k} \, R^{\tfrac{k}{2}}
       \sum_{\substack{s,t \\ \frac{R}{q} \ll st }}
       \frac{1}{(st)^k}\ll \sumd_{R\ll q^\delta} \frac{R^{1-\tfrac{k}{2}}}{q} \ll \frac{1}{q}.
\end{align*}
\end{proof}
\msubsection{Proof of Equation (\ref{Equation 6.3}) — Bound for}{S_{1,\mathrm{case}_2}}

Explicitly, we have
\begin{align*}
   S_{1,\mathrm{case}_2}
   := \frac{1}{\log q}\,\mathcal{K}\!
   \sumd_{R \ll q^\delta}\; \sum_{\substack{s,t \\ \tfrac{\sqrt{R}}{q} \ll st \ll \tfrac{R}{q}}}
   \frac{2\pi \mu(t)}{qst \phi(t)} V_{qs}(0,1;t)
   \sum_{\substack{r \\ (r,t)=1}} a_r \, W\!\left(\tfrac{r}{R}\right)
   e\!\left(\tfrac{r+1}{qst}\right)
   J_{k-1}\!\left(\tfrac{4\pi \sqrt{r}}{qst}\right).
\end{align*}
The key tool bounding $S_{1,\mathrm{case}_2}$ and $S_{1,\mathrm{case}_3}$ will be a Mellin transform representation of the exponential function, which allows us to separate the variables. 

\begin{lemma}\label{Lemma 6.3}
Let $\eta(x)$ be a smooth weight function such that
\[
\eta(x) =
\begin{cases}
1 & \text{if } x \in [a,b], \\[4pt]
0 & \text{if } x \notin (\tfrac{a}{2},2b),
\end{cases}
\]
where $a$ and $b$ are fixed positive constants. Define
\[
   \mathcal{M}_1(s) = \int_{0}^{\infty} \eta(y)\,
   e\!\left(\tfrac{yX}{j}\right)
   e\!\left(\pm \tfrac{\alpha \sqrt{yX}}{j}\right)
   y^{s-1}\, dy,
\]
where $\alpha\geq 0$ and $A, j,X\geq 1$ be fixed constants. For $X$ large enough we get 
\[
   \eta(y)\,
   e\!\left(\tfrac{yX}{j}\right)
   e\!\left(\pm \tfrac{\alpha\sqrt{yX}}{j}\right)
   = \frac{1}{2\pi i}\int_{(c)} \mathcal{M}_1(s)\,y^{-s}\,ds.
\]
Moreover, there exist constants $c_1,d_1$ such that
\[
   \mathcal{M}_1(c+iv) \ll
   \begin{cases}
      \sqrt{\tfrac{j}{X}}, & \text{if } \tfrac{d_1X}{j}\leq |v|\leq \tfrac{c_1X}{j}, \\[6pt]
      \dfrac{1}{\left(|v|+\tfrac{X}{j}\right)^A}, & \text{otherwise},
   \end{cases}
\]
where the implied constant depends on $A$ and $c$.
\end{lemma}

\begin{proof}
We write
\begin{align*}
\mathcal{M}_1(c+iv)&=\int_{0}^{\infty}\eta(y)e\left(\frac{yX}{j}\pm\frac{\alpha\sqrt {yX}}{j}\right)y^{c-1+iv}\ dy\\
    &=\int_{0}^{\infty} \eta(y)y^{c-1}e\left(\frac{yX}{j}\pm\frac{\alpha\sqrt {yX}}{j}+\frac{v}{2\pi}\log y\right)\ dy.
\end{align*}
Let $$g(y)=\frac{yX}{j}\pm\frac{\alpha\sqrt {yX}}{j}+\frac{v}{2\pi}\ \log y.$$
Therefore, we have
\begin{equation*}
    g'(y)= \frac{X}{j}\pm\frac{\alpha\sqrt {X}}{2j\sqrt y}+\frac{v}{2\pi y}.
\end{equation*}
\medskip

If $|v|\geq \frac{c_1X}{j}$ for some $c_1$, the $\frac{v}{2\pi y}$ term dominates and we do not get any stationary point. Therefore, using bounds for the exponential integral we get $\mathcal{M}_1(c+iv)\ll  1 /\left(|v|+\tfrac{X}{j}\right)^A$. We can show that there is no stationary point for $|v|\leq \frac{d_1X}{j}$ for some $d_1$ and by similar argument we get the same bound in this case as well.\\

If $\tfrac{d_1X}{j}\leq |v|\leq \tfrac{c_1X}{j}$, we need to analyze the saddle point at $y_0$. For sufficiently large $X$, 
we get $|g''(y_0)|\asymp \frac{X}{j}$. Hence, the corresponding bound is $\sqrt{\frac{j}{X}}$. 
\end{proof}
By Equation (\ref{J-Bessel expansion 2}), we have
\begin{equation*}
    J_{k-1}\!\left(\frac{4\pi \sqrt{r}}{qst}\right)
    = \sum_{l=0}^{\infty}
      \frac{(-1)^l}{l!\,(l+k-1)!}
      \left(\frac{2\pi \sqrt{r}}{qst}\right)^{2l+k-1}.
\end{equation*}
Define
\begin{equation*}
    \eta_l(y) \;=\; 
    \hat{\varphi}\!\left(\frac{\log (yR)}{\log q}\right)
    W(y)\, y^{\,l+\tfrac{k}{2}-\tfrac{1}{2}} .
\end{equation*}
Applying Lemma~\ref{Lemma 6.3} with $\alpha=0$, we obtain
\begin{equation*}
    \eta_l\!\left(\tfrac{r}{R}\right)\,
    e\!\left(\tfrac{r}{R}\cdot \tfrac{R}{qst}\right)
    = \frac{1}{2\pi i}
      \int_{(1)} \mathcal{M}_{1,l}(z)\,
      \left(\tfrac{R}{r}\right)^z \, dz,
\end{equation*}
where
\begin{equation*}
    \mathcal{M}_{1,l}(z)
    = \int_{0}^{\infty} \eta_l(y)\,
      e\!\left(\tfrac{yR}{qst}\right)\,
      y^{z-1}\, dy.
\end{equation*}
Thus,
\begin{align*}
   S_{1,\mathrm{case}_2}
   \ll \frac{1}{\log q}
       \sumd_{R \ll q^\delta}
       \sum_s
       \sum_{\substack{t \\ \tfrac{\sqrt{R}}{q} \ll st \ll \tfrac{R}{q}}}
       &\Biggl| 
         \frac{2\pi \mu(t)}{qst\,\phi(t)} \,
         V_{qs}(0,1;t) \,
         e\!\left(\tfrac{1}{qst}\right)
         \sum_{l=0}^{\infty}
         \frac{1}{l!\,(l+k-1)!}
         \left(\frac{2\pi \sqrt{R}}{qst}\right)^{2l+k-1} \\
   &\times
         \frac{1}{2\pi i}
         \int_{(1)} 
            \sum_{\substack{r \\ (r,t)=1}}
            \frac{\Lambda(r)}{r^{\tfrac{1}{2}+z}}
            \,\mathcal{M}_{1,l}(z)\, R^z \, dz 
       \Biggr|.
\end{align*}
We can write the sum over $r$ as 
\begin{equation*}
    \frac{1}{2\pi i}\int_{(1)}-\frac{L'}{L}\left(\frac{1}{2}+z,\chi_0\right)\mathcal{M}_{1,l}(z)R^z\ dz,
\end{equation*} 
where $\chi_0 $ is the trivial character modulo $ t$. Assuming GRH we can shift the contour to $\mathrm{}{\Re}(z)=\frac{1}{\log q}$ and pick up the pole at $z=\frac{1}{2}$. In particular,

    \begin{equation}\label{Equation 5.5z}
        \frac{1}{2\pi i}\int_{(1)}-\frac{L'}{L}\left(\frac{1}{2}+z,\chi_0\right)\mathcal{M}_{1,l}(z)R^z\ dz=\mathcal{M}_{1,l}\left(\frac{1}{2}\right)R^{\frac{1}{2}}+\int_{(\frac{1}{\log q})}\frac{L'}{L}\left(\frac{1}{2}+z,\chi_0\right)\mathcal{M}_{1,l}(z)R^z\ dz.
    \end{equation}
Let the contribution of the pole be $S_1, \mathrm{case_2}, Prin$ and the contribution of the integral be $S_1, \mathrm{case_2}, int$.\\
Now using Lemma \ref{Lemma 6.3} with $A=1$ and recalling that $k\geq 3$ odd, we get
\begin{align*}
   S_{1,\mathrm{case_2},Prin} &\ll \frac{1}{\log q}\sumd_{R\ll q^\delta} \sum_s \sum_{\substack{t \\  \frac{\sqrt{R}}{q}\ll st\ll \frac{R}{q}}}\frac{1}{qst}\sum_{l=0}^{\infty}\frac{1}{l!(l+k-1)!}\left(\frac{2\pi\sqrt{R}}{qst}\right)^{2l+k-1}
    \left(\frac{qst}{R}\right)R^{\frac{1}{2}}\\
    &=\frac{1}{\log q}\sumd_{R\ll q^\delta}\sum_s \sum_{\substack{t \\  \frac{\sqrt{R}}{q}\ll st\ll \frac{R}{q}}}\sum_{l=0}^{\infty}\frac{1}{l!(l+k-1)!}\ \frac{(2\pi\sqrt{R})^{2l+k-2}}{(qst)^{2l+k-1}} \\ &
    \ll \frac{1}{\log q}\sumd_{R\ll q^\delta}\frac{(\sqrt{R})^{k-2}}{q^{k-1}} \hspace{-.5 cm}\sum_{\substack{u\\ \frac{\sqrt R}{q} \ll u \ll\frac{R}{q} }}\frac{d(u)}{u^{k-1}}\ll  \sumd_{R\ll q^\delta}\frac{1}{q}\ll \frac{\log q}{ q}.
    \end{align*}
We now consider the contribution from $S_1, \mathrm{case_2}, int$. By applying the explicit formula (see \cite[Section~5.5]{IK}) and using that under GRH all non-trivial zeros lie on the critical line, and there are
$O(\log(q|t|))$ zeros up to height $|t|$ (see \cite[Sections~5.7--5.8]{IK}), we obtain that
\begin{equation}
    \frac{L'}{L}\left(\frac{1}{2}+\sigma+it,\chi\right)\ll \log^2(q+|t|).\label{eqn 6.6}
\end{equation}
Using Equation (\ref{eqn 6.6}) and Lemma \ref{Lemma 6.3}, we have
\begin{align}
    \int_{(\frac{1}{\log q})}\frac{L'}{L}\left(\frac{1}{2}+z,\chi_0\right)\mathcal{M}_{1,l}(z)R^z\ dz \nonumber &\ll
    R^{\frac{1}{\log q}}\ (\log q)^2 \int_{0}^{c_1\frac{R}{qst}}\sqrt{\frac{qst}{R}}\ dt \ +R^{\frac{1}{\log q}}\int_{c_1\frac{R}{qst}}^{\infty}\frac{\log^2 (t+q)}{t^2}\ dt\\
    &\ll \sqrt{\frac{R}{qst}} \log^2 q+\left(\frac{qst}{R}\right)\log^2 q \ll \sqrt{\frac{R}{qst}}\log^2 q,\label{equation 6.7}
\end{align}
since $st\ll \frac{R}{q}$.
Therefore, we obtain
\begin{align*}
    S_{1, \mathrm{case_2}, int}& \ll \frac{1}{\log q}\sumd_{R\ll q^\delta}\sum_s \sum_{\substack{t \\  \frac{\sqrt{R}}{q}\ll st\ll \frac{R}{q}}} \frac{1}{qst}\sum_{l=0}^{\infty}\frac{1}{l!(l+k-1)!}\left(\frac{2\pi\sqrt{R}}{qst}\right)^{2l+k-1} \log^2 q \left(\sqrt{\frac{R}{qst}}\right)\\
    & \ll \sumd_{R\ll q^\delta}\log q \sum_s \sum_{\substack{t \\  \frac{\sqrt{R}}{q}\ll st\ll \frac{R}{q}}} \frac{(\sqrt{R})^k}{(qst)^{k+\frac{1}{2}}}
     \ll \sumd_{R\ll q^\delta}\frac{q^\epsilon}{q}R^{\frac{1}{4}}\ll  q^\epsilon \frac{q^{\frac{\delta}{4} }}{q}.    
\end{align*}
Hence, for $\delta<4$ 
\begin{align*}
    S_{1,\mathrm{case_2}} \ll \frac{\log q}{q}+\frac{q^{\frac{\delta}{4}}q^\epsilon }{q}\ll \frac{1}{\log q}.
\end{align*}

\msubsection{Proof of Equation (\ref{Equation 6.4}) — Bound for}{S_{1,\mathrm{case}_3}}

Explicitly, we have the following.
\begin{align*}
   S_{1,\mathrm{case}_3}
   := \frac{1}{\log q}\,\mathcal{K}\!
   \sumd_{R \ll q^\delta}\;\sum_{\substack{s,t \\ \frac{\sqrt R}{q}\gg st}}
   \frac{2\pi \mu(t)}{qst \phi(t)} V_{qs}(0,1;t)
   \sum_{\substack{r \\ (r,t)=1}} a_r \, W\!\left(\tfrac{r}{R}\right)
   e\!\left(\tfrac{r+1}{qst}\right)
   J_{k-1}\!\left(\tfrac{4\pi \sqrt{r}}{qst}\right).
\end{align*}
Using Equation (\ref{J-Bessel expansion 1}), we obtain
\begin{equation*}
    J_{k-1}\left(\frac{4\pi \sqrt r}{qst}\right)=\frac{1}{2\sqrt 2 \pi} \left(\frac{qst}{\sqrt r}\right)^{\frac{1}{2}} \Re\left [w_k\left(\frac{4\pi \sqrt r}{qst}\right)e\left(\frac{2\sqrt r}{qst}-\frac{k}{4}+\frac{1}{8}\right)\right].
\end{equation*}
Similar to the previous case $S_{1,\mathrm{case}_2}$, we define
\begin{equation*}
    \eta(y):=\hat{\varphi}\left(\frac{\log yR}{\log q}\right) W(y) y^{-\frac{1}{4}} w_{k,1}\left(\frac{4\pi \sqrt{yR}}{qst}\right),
\end{equation*}
where $w_{k,1}$ can be $w_k$ or $\bar{w}_k$.
By Lemma \ref{Lemma 6.3} with $\alpha=2$, we have
\begin{equation*}
    \eta\left(\frac{r}{R}\right) e\left(\frac{r}{R} \frac{R}{qst}\right) e\left( \frac{2\sqrt R}{qst}\sqrt{\frac{r}{R}}\right)
    = \frac{1}{2\pi i} \int_{(1)} \mathcal{M}_{2,l}(z) \left(\frac{R}{r}\right)^z \, dz,
\end{equation*}
where
\begin{equation*}
    \mathcal{M}_{2,l}(y) = \eta(y)\, e\left(y \frac{R}{qst}\right) e\left( 2\sqrt{y} \frac{\sqrt R}{qst}\right).
\end{equation*}

\begin{align*}
     S_{1,\mathrm{case}_3} &\ll \frac{1}{\log q} \sumd_{R\ll q^\delta } \sum_s \sum_{\substack{t \\ st \ll \frac{\sqrt R}{q} }}\Biggl| 
     \frac{2\pi \mu(t)}{qst \phi(t)} V_{qs}(0,1;t) e\left(\frac{1}{qst}\right) 
     \left(\frac{qst}{\sqrt R}\right)^{\frac{1}{2}}
     \frac{1}{2\pi i} \int_{(1)} \sum_r \frac{\Lambda(r)}{r^{\frac{1}{2}+z}} \mathcal{M}_{2,l}(z) R^z \, dz \Biggr|.
\end{align*}
Similar to Equation (\ref{Equation 5.5z}), the integral over $z$ becomes
\begin{equation*}
    \frac{1}{2\pi i} \int_{(1)} -\frac{L'}{L}\left(\frac{1}{2}+z,\chi_0\right) \mathcal{M}_{2,l}(z) R^z \, dz
    = \mathcal{M}_{2,l}\left(\frac{1}{2}\right) R^{\frac{1}{2}} - \frac{1}{2\pi i} \int_{\left(\frac{1}{\log q}\right)} \frac{L'}{L}\left(\frac{1}{2}+z, \chi_0\right) \mathcal{M}_{2,l}(z) R^z \, dz.
\end{equation*}
Let the contribution from the pole be $S_{1,\mathrm{case}_3, \mathrm{Prin}}$ and the contribution from the integral be $S_{1,\mathrm{case}_3, \mathrm{int}}$. Then
\begin{align*}
    S_{1,\mathrm{case}_3, \mathrm{Prin}} &\ll \frac{1}{\log q} \sumd_{R\ll q^\delta} \sum_s \sum_{\substack{t \\ st \ll \frac{\sqrt{R}}{q} }} 
    \frac{1}{qst} \left(\frac{qst}{\sqrt R}\right)^{\frac{1}{2}} \left(\frac{qst}{R}\right) R^{\frac{1}{2}} \\ & \ll \frac{1}{\log q}  \sumd_{R\ll q^\delta} \sum_s \sum_{\substack{t \\ st \ll \frac{\sqrt{R}}{q} }} \frac{(qst)^{\frac{1}{2}}}{R^{\frac{3}{4}}} \ll \frac{\log q}{q}.
\end{align*}
By a similar calculation to Equation (\ref{equation 6.7}), we get
\begin{align*}
    \int_{(\frac{1}{\log q})} \frac{L'}{L}\left(\frac{1}{2}+z,\chi_0\right) \mathcal{M}_{2,l}(z) R^z \, dz &\ll \sqrt{\frac{R}{qst}} \log^2 q.
\end{align*}
Therefore,
\begin{align*}
    S_{1,\mathrm{case}_3, \mathrm{int}} &\ll \frac{1}{\log q} \sumd_{R\ll q^\delta} \sum_s \sum_{\substack{t \\ st \ll \frac{\sqrt R}{q} }} 
    \frac{1}{qst} \left(\frac{qst}{\sqrt R}\right)^{\frac{1}{2}} \sqrt{\frac{R}{qst}} \log^2 q   \ll \frac{q^\epsilon}{q} q^{\frac{\delta}{4}}.
\end{align*}
Since $\delta < 4$, we have
\begin{equation*}
    S_{1,\mathrm{case}_3} \ll \frac{q^\epsilon}{q} q^{\frac{\delta}{4}} + \frac{\log q}{q}\ll \frac{1}{\log q}.
\end{equation*}

\section{PROOF OF LEMMA \ref{Lemma 4.3} -- CONTRIBUTION FROM THE NON-TRIVIAL CHARACTERS}\label{Section 7}

Now, we will bound the remaining contribution from the non-trivial characters to complete the proof of Theorem \ref{Main Thm}.  
Recall that 
\begin{align*}
    S_2 \;=\; \frac{1}{\log q}\mathcal{K} 
    \sum_{s,t} \frac{2\pi}{qst\phi(t)}
    \sumstar_{m(t)} V_{qs}(m,1;t)
    \sum_{\substack{\chi(t) \\ \chi\neq\chi_0 }} \bar{\chi}(m)
    \sum_p \chi(p) a_p \,
    e\!\left(\frac{p+1}{qst}\right)
    J_{k-1}\!\left(\frac{4\pi \sqrt{p}}{qst}\right).
\end{align*}
We will consider sums over $p$, $s$, and $t$ in dyadic intervals of the form $P \leq p \leq 2P$, $T \leq t \leq 2T$, and $S \leq s \leq 2S$, denoted by $p\sim P$, $t \sim T$, and $s \sim S$, respectively.  
For nice smooth functions $U, X, Y \in C_c^\infty([\frac{3}{4},2])$, similar to equation (\ref{4.111}), we define
\begin{align}
    D(P,S,T) \hspace{-.1 cm}:&= \hspace{-.2 cm}
    \sum_{s} \sum_{t} \frac{1}{qst\phi(t)} 
    X\!\left(\frac{s}{S}\right) 
    Y\!\left(\frac{t}{T}\right)
    \hspace{-.1 cm}\sumstar_{m(t)} V_{qs}(m,1;t)
    \hspace{-.1 cm}\sum_{\substack{\chi(t) \\ \chi\neq\chi_0 }} \hspace{-.1 cm}\bar{\chi}(m)
    \hspace{-.1 cm}\sum_{p} \hspace{-.06 cm}\chi(p)a_p \, \nonumber \\ & \times 
    U\!\left(\frac{p}{P}\right)
    e\!\left(\frac{p+1}{qst}\right)
    J_{k-1}\!\left(\frac{4\pi \sqrt{p}}{qst}\right).
\end{align}
Next, we decompose $D(P,S,T)$ into three ranges depending on the transition behavior of the $J$-Bessel function and the exponential term:
\begin{equation}
   D(P,S,T) \;:=\; D(P,S,T)_{\mathrm{case}_1} \;+\; D(P,S,T)_{\mathrm{case}_2} \;+\; D(P,S,T)_{\mathrm{case}_3},
\end{equation}
where $ D(P,S,T)_{\mathrm{case}_1}$  is restricted to region $\frac{P}{q}\ll ST$, $D(P,S,T)_{\mathrm{case}_2}$ is restricted to region $\frac{\sqrt P}{q}\ll ST\ll \frac{P}{q}$ and lastly $D(P,S,T)_{\mathrm{case}_3}$ is restricted to $\frac{\sqrt P}{q}\gg ST$.  
Again we can assume $P \ll q^{\delta}$ in all three cases due to $\operatorname{supp} \hat{\varphi} \subset (-\delta,\delta)$.\\
\begin{lemma}
   Assume GRH. With the above notation, for $k\geq 3$ odd, the following estimates hold for any $\epsilon>0$: \vspace{-.5 cm}
   \begin{align}
      D(P,S,T)_{\mathrm{case}_1} &\ll q^\epsilon \frac{\sqrt P}{q^2S}\left(\frac{\sqrt P}{qST}\right)^{k-2}, 
         &&  
         \label{equation 7.6}\\[0.5em]
      D(P,S,T)_{\mathrm{case}_2} &\ll \frac{q^\epsilon\  P^{\frac{3}{4}} }{q^2S}, 
         && 
         \label{equation 7.7}\\[0.5em]
      D(P,S,T)_{\mathrm{case}_3} &\ll  \frac{q^\epsilon\  P^{\frac{3}{4}}}{q^2S}. 
         && 
         \label{equation 7.8}
   \end{align}
\end{lemma}

\subsection{Proof of Lemma \ref{Lemma 4.3}}
Note that if $k\geq 3$ odd and $P\ll q^\delta$, we get
\begin{align*}
    S_2&=\sumd_{P}  \sumd_{S} \sumd_{T} \Bigg[D(P,S,T)_{\mathrm{case}_1} \;+\; D(P,S,T)_{\mathrm{case}_2} \;+\; D(P,S,T)_{\mathrm{case}_3}\Bigg]
    \\&\ll\sumd_{P\ll  q^\delta} \sumd_{\substack{S, T \\ \frac{P}{q}\ll ST }}q^\epsilon \frac{\sqrt P}{q^2S}\left(\frac{\sqrt P}{qST}\right)^{k-2}+\sumd_{P\ll  q^\delta} \sumd_{\substack{S, T \\ \frac{\sqrt P}{q}\ll ST\ll \frac{P}{q} }}\frac{q^\epsilon\  P^{\frac{3}{4}}}{q^2S}+\sumd_{P\ll q^\delta} \sumd_{\substack{S, T \\ ST\ll \frac{\sqrt{P}}{q} }} \frac{q^\epsilon\  P^{\frac{3}{4}}}{q^2S}\\&
    \ll\frac{q^\epsilon q^{\frac{\delta}{2}}}{q^2}+\frac{q^\epsilon\  q^{\frac{3\delta}{4}}}{q^2}+\frac{q^\epsilon\  q^{\frac{3\delta}{4}}}{q^2}.
\end{align*}
Therefore, for $\delta<\frac{8}{3}$, we have $S_2\ll \frac{1}{\log q}$, as desired.

\msubsection{Proof of Equation (\ref{equation 7.6}) — Bound for}{D(P,S,T)_{\mathrm{case}_1}}
For case 1, we have $ \frac{P}{q}\ll ST$, and explicitly, we write
\begin{align*}
    D(P,S,T)_{\mathrm{case}_1} &:= \frac{1}{\log q}\mathcal{K} 
   \sum_s \sum_t\frac{2\pi}{qst\phi(t)}
    X\!\left(\frac{s}{S}\right) 
    Y\!\left(\frac{t}{T}\right)
    \sumstar_{m(t)} V_{qs}(m,1;t)
    \sum_{\substack{\chi(t) \\ \chi\neq\chi_0 }} \bar{\chi}(m)
    \sum_{p} \chi(p)a_p\\& \times \,
    U\!\left(\frac{p}{P}\right)
    e\!\left(\frac{p+1}{qst}\right)
    J_{k-1}\!\left(\frac{4\pi \sqrt{p}}{qst}\right).\nonumber \hspace{1 cm}
\end{align*} 
Using the series expansion of the $J$-Bessel, in Equation (\ref{J-Bessel expansion 2}) of Lemma \ref{Lemma J-Bessel Function} we get
\begin{align}
      D(P,S,T)_{\mathrm{case}_1}
      &\ll \frac{1}{\log q}
      \sum_{s}\sum_{t}
      \Bigg|\frac{1}{qst\phi(t)}X\!\left(\frac{s}{S}\right) 
      Y\!\left(\frac{t}{T}\right)
      \sumstar_{m(t)}V_{qs}(m,1,t)
      \sum_{\substack{\chi(t) \\ \chi\neq \chi_0 }}\bar{\chi}(m)
      \sum_{p} \chi(p)a_p 
      U\!\left(\frac{p}{P}\right) \nonumber\\
      &\hspace{2.4 cm} \times e\!\left(\frac{p+1}{qst}\right) 
      \sum_{l=0}^{\infty} 
      \frac{\left(\tfrac{2\pi \sqrt p}{qst}\right)^{2l+k-1}}{l!\,(l+k-1)!}
      \Bigg| \nonumber\\
      &\ll \frac{1}{\log q}\frac{1}{qS}
      \sum_{l=0}^{\infty} 
      \frac{1}{l!\,(l+k-1)!}
      \left(\frac{2\pi\sqrt P}{qST}\right)^{2l+k-1} 
      \sum_{s}\sum_{t}\Bigg|
      \frac{1}{t\phi(t)}X\!\left(\frac{s}{S}\right) 
      Y\!\left(\frac{t}{T}\right)\Bigg|(W_1W_2)^{1/2}\label{Equation 6.7x},
\end{align}
where for $i=1,2$,
\begin{align*}
    W_{i}&=\sum_{\substack{x(t) \\ (x(x+qs),t)=1 }}|T_i(x)|^2, \\
    T_1(x;l)& = \sumstar _{m(t)}e\!\left(\frac{m\bar{x}}{t}\right)S_{1,l}(m), \\
    S_{1,l}(m)&=\sum_{\substack{\chi(t) \\ \chi\neq \chi_0 }}\bar{\chi}(m)
    \sum_{p} a_p U\!\left(\frac{p}{P}\right)\chi(p) 
    e\!\left(\frac{p}{qst}\right)\left(\sqrt{\frac{p}{P}}\right)^{2l+k-1}, \\
    T_2(x)&=e\!\left(-\frac{\overline{x+qs}}{t}\right)e\left(\frac{1}{qst}\right).
\end{align*}
We have
\begin{equation}\label{Eqn 6.7x}
    W_{2} =\sum_{\substack{x(t) \\ (x(x+qs),t)=1 }}
    |T_2(x;l)|^2 
     \leq \phi(t).
\end{equation}
Moreover, 
\begin{align*}
    W_{1}
    &\leq \sum_{x(t)}
    \left|\sumstar_{m(t)}e\!\left(\frac{mx}{t}\right)S_{1,l}(m)\right|^2\\ &=t\sumstar_{m(t)}
    \left|\sum_{\substack{\chi(t) \\ \chi\neq \chi_0 }}\bar{\chi}(m)
    \sum_{p}  \chi(p)U\!\left(\frac{p}{P}\right) 
    e\!\left(\frac{p}{qst}\right)\left(\sqrt {\frac{p}{P}}\right)^{2l+k-1}
    \frac{\log p}{\sqrt p}\,
    \hat{\varphi}\!\left(\frac{\log p}{\log q}\right)\right|^2.
\end{align*}
Using the orthogonality of the Dirichlet characters, we get
\begin{align*}
    W_{1} 
    &\leq  t\phi(t)
    \sum_{\substack{\chi(t) \\ \chi\neq \chi_0 }}
    \left|\sum_{p}  
    \frac{\chi(p)}{\sqrt p}\log p \,
    e\!\left(\frac{p}{qst}\right)\left(\sqrt {\frac{p}{P}}\right)^{2l+k-1}
    U\!\left(\frac{p}{P}\right)
    \hat{\varphi}\!\left(\frac{\log p}{\log q}\right)\right|^2 \\
    &= t\phi(t)
    \sum_{\substack{\chi(t) \\ \chi\neq \chi_0 }}
    \left|\sum_{p}\frac{\chi(p)\,\log p}{\sqrt p} \,
    e\left(\frac{p}{P}\cdot \frac{P}{qst}\right)
    U\!\left(\frac{p}{P}\right)
    \left(\sqrt {\frac{ p}{ P}}\right)^{2l+k-1}
    \hat{\varphi}\!\left(\frac{\log(p/P)+\log P}{\log q}\right)\right|^2.
\end{align*}
Let
\[
    \psi(y)=e\!\left(y\cdot \tfrac{P}{qst}\right)
    U(y)(\sqrt{y})^{2l+k-1}
    \hat{\varphi}\!\left(\frac{\log y+\log P}{\log q}\right).
\]
Using condition $\frac{P}{q}\ll ST$, we obtain $|\psi^3(x)|\ll (1+l)^3$.
Then, using Lemma \ref{Lemma 2.4}, for non-trivial characters, we get
\[
    \sum_{p}\frac{\chi(p)\log p}{\sqrt p}\psi\!\left(\frac{p}{P}\right)
    \ll\log^2 q\,(1+l)^3.
\]
Therefore, for $k\geq 3$ odd, we get
\begin{align*}
    D(P,S,T)_{\mathrm{case}_1}
    &\ll \frac{q^\epsilon}{qS}\sum_{l=0}^{\infty}
    \frac{(l+1)^3}{l!\,(l+k-1)!}
    \left(\frac{2\pi\sqrt P}{qST}\right)^{2l+k-1}
    \sum_{s}\sum_{t}
    \frac{1}{t\phi(t)}\,t\phi(t)\,
    \Bigg|X\!\left(\frac{s}{S}\right) 
    Y\!\left(\frac{t}{T}\right)\Bigg|  \\
    &\ll \frac{q^\epsilon}{qS}\sum_{l=0}^{\infty}
    \frac{(l+1)^3}{l!\,(l+k-1)!}\,
    \left(\frac{2\pi\sqrt P}{qST}\right)^{2l+k-1} 
    ST \\&\ll q^\epsilon \frac{\sqrt P}{q^2S}\left(\frac{\sqrt P}{qST}\right)^{k-2}.
\end{align*}
\msubsection{Proof of Equation (\ref{equation 7.7}) — Bound for}{D(P,S,T)_{\mathrm{case}_2}}

For case 2, we have $\frac{\sqrt P}{q}\ll ST\ll \frac{P}{q}$. Using the series expansion of the $J$-Bessel function from Equation (\ref{J-Bessel expansion 2}) of Lemma \ref{Lemma J-Bessel Function}, and following similar calculations in $D(P,S,T)_{\mathrm{case}_1}$ under the same notation as in Equation \ref{Equation 6.7x}, we obtain
\begin{equation*}
     D(P,S,T)_{\mathrm{case}_2}\ll \frac{1}{\log q}\ \frac{1}{qS}
      \sum_{l=0}^{\infty} 
      \frac{1}{l!\,(l+k-1)!}
      \left(\frac{2\pi\sqrt P}{qST}\right)^{2l+k-1} 
      \sum_{s}\sum_{t}\Bigg|
      \frac{1}{t\phi(t)}X\!\left(\frac{s}{S}\right) 
      Y\!\left(\frac{t}{T}\right)(W_1W_2)^{1/2}
      \Bigg|,
\end{equation*}
where $W_{2}\leq \phi(t)$. Suppose $U_1\in C_c^\infty([\frac{1}{2},\frac{5}{2}])$, and $U_1(x)=1$ on $[1,2]$.
\[
 W_{1}\ll t\phi(t)\ \sum_{\substack{\chi(t) \\ \chi\neq\chi_0}}
 \left|\sum_{p}\frac{\chi(p)\log p}{\sqrt p}\, e\!\left(\frac{p}{P}\cdot \frac{P}{qst}\right) 
 U\!\left(\frac{p}{P}\right)U_1\left(\frac{p}{P}\right)\,
 \left( \sqrt{\frac{ p}{ P}}\right)^{2l+k-1}
 \hat{\varphi}\!\left(\frac{\log \tfrac{p}{P}+\log P}{\log q}\right)\right|^2.
\]
Define
\begin{equation*}
    \psi(y)=U(y)\,(\sqrt{y})^{2l+k-1}\hat{\varphi}\!\left(\frac{\log y+\log P}{\log q}\right)\frac{1}{y}.
\end{equation*}
From Lemma \ref{Lemma 6.3}, we have
\begin{equation*}
    U_1(y)\, e\!\left(y\cdot \frac{P}{qst}\right)=\frac{1}{2\pi i}\int_{(1)}\mathcal{M}_2(s)y^{-s}\,ds,
\end{equation*}
where
\begin{equation*}
\mathcal{M}_2(1+iv)=\int_{0}^{\infty}U_1(y)\, e\!\left(\frac{yP}{qst}\right)y^{iv}\,dy.
\end{equation*}
Using Lemma \ref{Lemma 2.4} and the bound of $\mathcal{M}_1(1+iv)$ in Lemma \ref{Lemma 6.3}, we obtain the following.
\begin{align*}
W_{1}&\ll t\phi(t)\ \sum_{\substack{\chi(t) \\ \chi\neq\chi_0}}
        \left|\sum_{p}\frac{\chi(p)\log p}{\sqrt p}\psi\left(\frac{p}{P}\right)\int_{(1)}\mathcal{M}_2(z)\left(\frac{P}{p}\right)^{z-1}\,dz\right|^2 \\
    &\ll t\phi(t)\ \sum_{\substack{\chi(t) \\ \chi\neq\chi_0}}
    \Bigg[\int_{\mathbb{R}}\!|\mathcal{M}_2(1+iw)|\,
    \Bigg|\sum_{p}\frac{\log p}{p^{\frac{1}{2}+iw}}\chi(p)\psi\left(\frac{p}{P}\right)\, \Bigg|\, dw\Bigg]^2 \\
&\ll t^2\phi(t)\Bigg[\int_{\mathbb{R}}|\mathcal{M}_2(1+iw)| \,\log^2(q+|w|)(1+l)^3\,dw\Bigg]^2 \ll t^2\phi(t)\ q^\epsilon \left(\sqrt{\frac{P}{qST}} (1+l)^3 \right)^2.
\end{align*}
Since $\frac{\sqrt P}{q}\ll ST\ll \frac{P}{q}$,
\begin{align*}
D(P,S,T)_{\mathrm{case}_2}& \ll \frac{q^\epsilon}{qS}\sum_{l=0}^{\infty}\frac{(l+1)^3}{l!(l+k-1)!} \left(\frac{2\pi\sqrt P}{qST}\right)^{2l+k-1}\hspace{-.3 cm}\sqrt{\frac{P}{qST}}\sum_{s}\sum_{t}\frac{1}{t\phi(t)}\, t\phi(t)\, \Bigg|X\!\left(\frac{s}{S}\right)Y\!\left(\frac{t}{T}\right) \Bigg| \\
     & \ll \frac{q^\epsilon}{qS} \left(\frac{\sqrt P}{qST}\right)^{k-1} ST \sqrt{\frac{P}{qST}}\ll \frac{q^\epsilon}{q^2S} \left(\frac{\sqrt P}{qST}\right)^{k-\frac{3}{2}} P^{\frac{3}{4}}\ll \frac{q^\epsilon}{q^2S} P^{\frac{3}{4}},
\end{align*}
 as desired.

\msubsection{Proof of Equation (\ref{equation 7.8}) — Bound for}{D(P,S,T)_{\mathrm{case}_3}}

For case 3, we have $\frac{\sqrt P}{q}\gg ST$. Explicitly, we write
\begin{align*}
    D(P,S,T)_{\mathrm{case}_3} &:= \frac{1}{\log q}\mathcal{K} 
    \sum_{s}\sum_{t}\frac{2\pi}{qst\phi(t)}
    X\!\left(\frac{s}{S}\right) 
    Y\!\left(\frac{t}{T}\right)
    \sumstar_{m(t)} V_{qs}(m,1;t)
    \sum_{\substack{\chi(t) \\ \chi\neq\chi_0 }} \bar{\chi}(m)
    \sum_{p} \chi(p)a_p \,\\& \times
    U\!\left(\frac{p}{P}\right)
    e\!\left(\frac{p+1}{qst}\right)
    J_{k-1}\!\left(\frac{4\pi \sqrt{p}}{qst}\right).\nonumber 
\end{align*}  
By Equation (\ref{J-Bessel expansion 1}) of Lemma \ref{Lemma J-Bessel Function} we write
\begin{equation*}
    J_{k-1}\left(\frac{4\pi \sqrt r}{qst}\right)
    =\frac{1}{2\sqrt{2}\pi} \left(\frac{qst}{\sqrt r}\right)^{\frac{1}{2}}
    \Re\left[w_k\!\left(\frac{4\pi \sqrt r}{qst}\right)
    e\!\left(\frac{2\sqrt r}{qst}-\frac{k}{4}+\frac{1}{8}\right)\right].
\end{equation*}
Again, following the proof of $D(P,S,T)_{\mathrm{case}_1}$ with $w_{k,1}$ as $w_k$ or $\bar{w}_k$, we obtain
\begin{align*}
       D(P,S,T)_{\mathrm{case}_3}& \ll \sum_{s}\sum_{t}\Bigg|\frac{1}{qst\phi(t)}X\!\left(\frac{s}{S}\right) 
    Y\!\left(\frac{t}{T}\right)\sumstar_{m(t)}V_{qs}(m,1;t)\sum_{\substack{\chi(t) \\ \chi\neq\chi_0 }}\bar{\chi}(m)\sum_p\chi(p)a_p U\!\left(\frac{p}{P}\right)e\!\left(\frac{p+1}{qst}\right)  \\
    & \hspace{1.4 cm}\times \left(\frac{qst}{\sqrt P}\right)^{\frac{1}{2}}\frac{P^{\frac{1}{4}}}{p^{\frac{1}{4}}} w_{k,1}\!\left(\frac{4\pi\sqrt p}{qst}\right)e\!\left(\frac{2\sqrt p}{qst}\right)\Bigg|\\
    &\ll \frac{1}{qS}\frac{(qST)^{\frac{1}{2}}}{P^{\frac{1}{4}}}\sum_{s}\sum_{t}\frac{1}{t\phi(t)}\Bigg|X\!\left(\frac{s}{S}\right) 
    Y\!\left(\frac{t}{T}\right)\Bigg|(W_1W_2)^{\frac{1}{2}},
\end{align*}

where
\begin{align*}
    W_{1}&=\sum_{\substack{x(t) \\ (x(x+qs),t)=1 }}|T_1(x)|^2,\\
    T_1(x;l)& = \sumstar _{m(t)}e\!\left(\frac{m\bar{x}}{t}\right)S_{1,l}(m),\\
    S_{1,l}(m)&=\sum_{\substack{\chi(t) \\ \chi\neq\chi_0 }}\bar{\chi}(m)\sum_p a_p U\!\left(\frac{p}{P}\right)\chi(p) e\!\left(\frac{p}{qst}\right)w_{k,1}\!\left(\frac{4\pi\sqrt p}{qst}\right) e\!\left(\frac{2\sqrt p}{qst}\right)\left(\frac{P}{p}\right)^{\frac{1}{4}}.
\end{align*}
$W_2$ is defined as in Equation (\ref{Eqn 6.7x}). Following the same arguments as in $D(P,S,T)_{\mathrm{case}_2}$, we add the factor $U_1(\frac{p}{P})$ and obtain 

\begin{align*}
    W_{1} & \ll t\phi(t)\ \sum_{\substack{\chi(t) \\ \chi\neq\chi_0 }}\Bigg|\sum_{p}\frac{\chi(p)\log p}{\sqrt p}\, e\!\left(\frac{p}{P}\cdot \frac{P}{qst}\right)\, U\!\left(\frac{p}{P}\right)\,  U_1\!\left(\frac{p}{P}\right)\hat{\varphi}\!\left(\frac{\log \frac{p}{P}+\log P}{\log q}\right)\frac{P^{\frac{1}{4}}}{p^{\frac{1}{4}}} \\
    &\hspace{1.6 cm} \quad \times W_{k,1}\!\left(\frac{4\pi \sqrt p}{qst}\right)e\!\left(\frac{2\sqrt p}{qst}\right)\Bigg|^2.
\end{align*}
Let
\[
\psi(y)=\frac{1}{y}U(y)\, \frac{1}{y^{\frac{1}{4}}}\, \hat{\varphi}\!\left(\frac{\log y+\log P}{\log q}\right)\, W_{k,1}\!\left(\frac{4\pi \sqrt{yP}}{qst}\right).
\]
By Lemma \ref{Lemma 6.3}, with $\alpha =2$ we get
\begin{equation*}
    U_1(y)\, e\!\left(y\cdot \frac{P}{qst}\right)e\!\left( \frac{2\sqrt{yP}}{qst}\right)=\frac{1}{2\pi i}\int_{(1)}\mathcal{M}_2(s)y^{-s}\, ds.
\end{equation*}
By Lemma \ref{Lemma 2.4} and Lemma \ref{Lemma 6.3}, for $\tfrac{\sqrt{P}}{q}\gg ST$, we get
\begin{align*}
    W_{1}&\ll t\phi(t)\ \sum_{\substack{\chi(t) \\ \chi\neq\chi_0 }}\left|\sum_{p}\frac{\chi(p)\log p}{\sqrt p}\, \psi\!\left(\frac{p}{P}\right)\int_{(1)}\mathcal{M}_2(z)\left(\frac{P}{p}\right)^{z-1}\, dz\right|^2\\
    &\ll t^2\phi(t)\Bigg[\int_{\mathbb{R}}|\mathcal{M}_2(1+iw)|\, \log^2(q+|w|) \ dw\Bigg]^2\ll t^2\phi(t)q^{\epsilon}\left(\sqrt{\frac{P}{qST}}\right)^2.
\end{align*}
Hence, using $\tfrac{\sqrt{P}}{q}\gg ST$, we get
\begin{align*}
 D(P,S,T)_{\mathrm{case}_3}&\ll  q^{\epsilon}\frac{(qST)^{\frac{1}{2}}}{P^{\frac{1}{4}}}\cdot \frac{1}{qS}\sum_{s}\sum_{t}\frac{1}{t\phi(t)}\, t\phi(t) \Bigg|X\!\left(\frac{s}{S}\right) 
    Y\!\left(\frac{t}{T}\right)\Bigg|\sqrt{\frac{P}{qST}}\\
    &\ll  \frac{q^{\epsilon}P^{\frac{1}{4}}}{qS}\, ST \ll \frac{q^\epsilon\  P^{\frac{3}{4}}}{q^2S},
\end{align*}
as desired.\\

\section*{FUNDING} This work was supported by the National Science Foundation [grant number DMS-2502599].

\section*{ACKNOWLEDGMENTS} The author is very grateful to Professor Vorrapan Chandee for her
guidance throughout the making of this paper.

 \bigskip

\noindent Mathematics Department, Kansas State University, Manhattan, KS 66503\\
\noindent Email address: \texttt{arijitpaul@ksu.edu} 

\end{document}